\documentclass[a4paper, 10pt, reqno, final]{article}

\usepackage[T1]{fontenc}		     
\usepackage[utf8]{inputenc}			 
\usepackage[american]{babel}
\usepackage{palatino}
\usepackage{float}
\usepackage{caption}
\usepackage{subcaption}

\usepackage{amsmath}
\usepackage{amssymb}
\usepackage{amsthm}
	\theoremstyle{plain}
		\newtheorem{mainthm}{\textsc{Theorem}}		
				
				\newtheorem{thm}{Theorem}[section]

				\newtheorem{lem}[thm]{Lemma}

	\theoremstyle{definition}
		\newtheorem{defn}[thm]{Definition}	
						
					\theoremstyle{remark}
		\newtheorem{rem}[thm]{Remark}		
				\newtheorem{note}[thm]{Notation}		
				\numberwithin{equation}{section}
				\usepackage{mathtools}		
				\usepackage{mathrsfs}		
				
\usepackage{eucal}			

\usepackage{braket}			
\usepackage[all,pdf]{xy}		

\usepackage{mparhack}		
\usepackage{relsize}			

\usepackage{a4wide}		

\usepackage{booktabs}		
\usepackage{multirow}		
\usepackage{caption}		
\usepackage{rotating}

\usepackage{varioref}		
\usepackage{footmisc}

\usepackage{graphicx}		
\usepackage{epsfig}

\usepackage{enumerate}		
\usepackage[colorlinks=true, linkcolor=black, citecolor=black,
urlcolor=blue]{hyperref}		
\mathtoolsset{showonlyrefs}

\newcommand{\trasp}[1]{{#1}^\mathsf{T}}	
\newcommand{\itrasp}[1]{{#1}^\mathsf{-T}}	
\newcommand{\R}{\mathbb{R}}		
\newcommand{\ZZ}{\mathbf{Z}}
\newcommand{\U}{\mathbf{U}}		

\newcommand{\igeo}{\iota^{\scriptscriptstyle{\mathrm{geo}}}}
\DeclareMathOperator{\sgn}{sgn}		

\newcommand{\Id}{I}

\def\sqm1{\sqrt{-1}}

\def\={\cong}
\def\>{\supset}
\def\<{\subset}

\def\12{\dfrac{1}{2}}
\def\0{^{\circ}}

\def\RR{{\mathbb R}}

\def\ZZ{{\mathbb Z}}

\def\rr{{\mathbb R}}

\def\Aa{{\mathcal A}}

\def\A{\Aa}

\def\R{\RR}

\def\Z{\ZZ}

\DeclareMathOperator{\morse}{m^-}
\DeclareMathOperator{\comorse}{m^+}

\DeclareMathOperator{\diag}{diag}

 \DeclareMathOperator{\Lag}{\Lambda}

\DeclareMathOperator{\Sp}{Sp}

\newcommand{\ga}{\gamma}
\newcommand{\gm}{\gamma}

\newcommand{\lmd}{\lambda}
\newcommand{\zt}{\zeta}

\newcommand{\vep}{\varepsilon}

\newcommand{\vr}{\varrho}

\newcommand{\xd}{\dot{x}}

\newcommand{\bh}{\widehat B }

\newcommand{\X}{\mathcal{X}}
\newcommand{\Xh}{\widehat{\mathcal{X}}}
\newcommand{\E}{\mathcal{E}}
\newcommand{\I}{\mathcal{I}}

\newcommand{\Mh}{\mathbf{M}}

\newcommand{\ld}{L_D}
\renewcommand{\ln}{L_N}

\title{Morse Index for Homothetic motions  in the gravitational $n$-body problem}
\author{Yuwei Ou\thanks{The author is supported by NSFC(No.12371192), the Young Taishan Scholars Program of Shandong Province(Grant No. tsqn202312055) and the Qilu Young
Scholar Program of Shandong University.}, Alessandro Portaluri
\thanks{The author is partially supported by INDAM-PRIN Project: N. 2022FPZEES
“Stability in Hamiltonian dynamics and beyond”.} }
\date{\today}

\begin{document}
 \maketitle

\begin{abstract}
In the gravitation $n$-body Problem, a homothetic orbit is a special solution of the Newton's Equations of motion, in which each body moves along a straight line through the center of mass and forming at any time  a central configuration. In 2020, Portaluri et al. proved that under a spectral gap condition on the limiting central configuration, known in literature as non-spiraling or [BS]-condition, the Morse index of an asymptotic colliding motion is finite.  Later Ou et al. proved  this result for other classes of unbounded motions, e.g. doubly asymptotic motions (e.g. doubly homothetic motions).

In this paper we prove that for a homothetic motion, irrespective of how large the index of the  limiting central configuration and how  large  the energy level is,  the following alternative holds: if the non-spiraling condition holds then the Morse index is 0 otherwise it is infinite. 
\end{abstract}

%


\section{Introduction and main results} \label{sec: intro}

The gravitational $n$-body problem describes the motion of $n \ge 1$ point mass particles $q_1, \ldots, q_n \in \R^d$ (here $d \ge 1$) having  masses  $m_i>0$ for $i=2, \ldots, n$ and moving under the self-interacting gravitation potential. The corresponding Newton's Equations of motion, can be written as follows:
\begin{equation}\label{eq:NewtonINTRO}
M\,\ddot q = \nabla U(q),
\end{equation}
where $M$ is the block diagonal matrix defined as $M:=[M_{ij}]_{i,j=1}^n$ with $M_{ij}:= m_j \delta_{ij} \Id_d$, where $\Id_d$ is the $d$-dimensional identity matrix, whilst  the {\sc potential function $U$} is given by
\[
U(q)= \sum_{\substack{i, j=1\\[0.3pt] i< j}}^n\dfrac{m_i \, m_j}{|q_i-q_j|}
\]
and finally $\nabla$ denotes the Euclidean gradient.
Since the center of mass  is a conserved quantity, wlog the {\sc configuration space} of the system is the $N=d(n-1)$-dimensional subspace defined by
\[
\X:=\Set{ q \in \rr^{dn}| \sum_{i=1}^n m_i q_i =0 }.
\]
For each pair of indices $i, j \in \Set{1, \ldots, n}$,  we let $\Delta_{i,j}=\Set{q:=(q_1, \ldots, q_n) \in (\R^d)^n| q_i=q_j}$ be the {\sc collision set of the $i$-th and $j$-th particles} and we let
\[
\Delta:= \bigcup_{\substack{i,j=1\\ i \neq j}}^n \Delta_{i, j}
\]
be the {\sc collision set in $\X$.} It turns out that $\Delta$ is a {\sc cone} with vertex is located at the origin $o \in \Delta$ which corresponds to the {\sc total collision} or {\sc total collapse} of the system (since the center of mass has been fixed at the origin). The space of {\sc collision free configurations} is denoted by
\[
 \Xh:=\X\setminus \Delta.
\]
Denoting by $T\X$ the tangent bundle of $\X$, whose elements are denoted by $(q,v)$ with $q\in X$ and $v$ a tangent vector at $q$, the Lagrangian function $L : T\X \to [0, +\infty)\cup\{+\infty\}$ of the system is
\begin{equation}\label{eq:lagrangianaINTRO}
 L(q,v)=K(v)+ U(q), \quad \textrm{ where } \quad K(v):=\dfrac{1}{2}|v|_M^2 :=  \langle  Mv, v \rangle
\end{equation}
the term $K$ is the {\sc kinetic energy} of the system. It is well-known (cfr. \cite{AZ94} and references therein) that the   genuine solutions of the Newton's Equation are up to standard bootstrap arguments critical points of the $\mathscr C^2$-regular  {\sc  Lagrangian action functional}
\begin{equation}
\mathcal A: W^{1,2}([t_1, t_2],  \Xh)\to \R \quad \textrm{ defined by } \quad \mathcal A(q):=\int_{t_1}^{t_2} L(q(t), \dot{q}(t)) dt.
\end{equation}
It is well-known that homothetic motions as well as other classes of motions like the class of total parabolic and hyperbolic motions not  experiencing any collision or non-collision singularity both in the future or the past, are generated by the action of $\R$ or $\R^+$ by the {\sc central configurations}, (cfr.  \cite{BHPT20, HOY21} and references therein), namely critical points of the restriction of the potential $U$ to the inertia ellipsoid. We are entitled to introduce the following definitions.
\begin{defn}\label{central confi}
let $\I(q):= \langle Mq, q \rangle $ be denote the {\sc moment of inertia} of the system and let $\E:= \{q \in \X: \I(q)=1 \}$ be the set of {\sc normalized configurations}.\footnote{We observe that the set $\E$ is the unit sphere in the mass norm $|\cdot|_M$ and it is usually termed {\sc inertia ellipsoid}.}
 We term {\bf central configuration} a critical point of the restriction   of $U|_{\E}$.
\end{defn}
In \cite{HOY21}, the authors consider the so-called {\sc doubly asymptotic motion}, whose  simplest example is represented by a  homothetic solution produced by the central configuration.
\begin{defn}
A {\bf homothetic motion} is a solution $q$   of the Newton's Equations of motion given at Equation~\eqref{eq:NewtonINTRO}, pointwise  defined  by $q(t)=r(t)s_0$ where $s_0$ is a central configuration and $t \mapsto r(t)$ is a positive real smooth function.\footnote{From a physical viewpoint the function $r$ is responsible of a dilation or a shrink of the  configuration without neither changing its shape nor rotating. Moreover it is easy to check that $r$ is a solution of the 1D Schr\"odinger's equation under a gravitational central force field.}
\end{defn}
For $T^{\pm} \in \rr$, let us consider a {\sc doubly  homothetic solution} $q$, namely a homothetic solution of the Equation~\eqref{eq:NewtonINTRO} such that $\displaystyle \lim_{t \to T^{\pm}} q(t)= 0$. Given $[t_1, t_2] \subset (T^-,T^+)$, the restriction  $q|_{[t_1, t_2]}$ is a not colliding solution of the Newton's equations. In particular $q|_{[t_1, t_2]}$ is a  critical point of the action functional $\A$ in $W_0^{1,2}([t_1, t_2], \Xh)$. In particular, such a restriction has a well-defined {\sc Morse index},  denoted by $\morse (q; t_1, t_2)$ and defined as the  maximal dimension of the negative spectral space of the second variation (quadratic form) of the Lagrangian action functional on $W_0^{1,2}([t_1, t_2], \Xh)$. Following authors in \cite{HOY21}, we introduce the   following definition.
\begin{defn}\label{dfn: Morse index}
Let $\{t^{\pm}_k\}_{k \in \Z^+}$ be two sequences of instants satisfying $ T^- < t^-_k< t^+_k < T^+$ and such that $\displaystyle \lim_{k \to \infty} t^{\pm}_k = T^{\pm}$. We define the {\bf Morse index} of the
\begin{itemize}
	\item {\bf Doubly  homothetic solution} $q \in \mathscr C^2((T^-, T^+), \Xh)$ as follows
\begin{equation}\label{dfn: morse index}
\morse (q; T^-, T^+)=\lim_{k \to \infty}\morse (q; t^-_k, t^+_k).
\end{equation}
\item {\bf Homothetic solution} $q \in \mathscr C^2([0, T^+), \Xh)$ starting from the  central configuration $s_0$ as
\begin{equation} \label{dfn: 2morse index}
\morse (q; 0, T^+)=\lim_{k \to \infty}\morse (q; 0, t^+_k).
\end{equation}
\end{itemize}
\end{defn}
 \begin{rem}
 It is worth observing that as a direct consequence of the  following monotonicity property  of the Morse index (see \cite{CH53} or \cite{HWY})
\begin{equation*}
\morse (q; t_1,t_2)\le \morse (q; t^*_1,t^*_2), \qquad  \textrm{ for } \qquad   t^*_1\le t_1  \textrm{ and }    t_2\le t^*_2
\end{equation*}
it holds that the Morse indices appearing at Definition~\ref{dfn: Morse index}
\[
\morse (q; T^-, T^+) \qquad \textrm{ and } \qquad  \morse (q; 0, T^+)
\]
 exist and are well-defined, meaning that they are independent on the choice of the sequences $\{t^{\pm}_k\}$.
\end{rem}
Given an arbitrary configuration $q \in \X$, we introduce the polar coordinates $(r, s) \in [0, +\infty) \times \E $ as follows
\[
r= \sqrt{\mathcal{I}(q)} \qquad  s= \dfrac{q}{r}=(s_1 \ldots, s_n) \in \E.
\]
\begin{rem}
It is worth observing that, in general, the angular part $s$ of a total colliding solution doesn't converge to a fixed central configuration. However in the case  of a doubly  homothetic motion of the form
$q(t)=r(t)s_{0}$, we get that  $s(t)=s_{0}$ for every $ t\in(T^{-}, T^{+})$. In particular  the limit of the angular part exists. \footnote{In the case of   singular homothetic motions, this is a special case of the class of $s_0$-asymptotic motions considered in \cite{BHPT20}.}
\end{rem}
The next step is to introduce a spectral gap condition on  the limiting central configuration $s_0$ that will play a crucial role. First of all, we observe that by a straightforward computation a critical point $s_0$ of the restriction of the potential function on the  set of normalized configurations, can be characterized as a solution of the following algebraic equation
\begin{equation}\label{eq: CC}
\nabla U|_{\E}(s)= \nabla U(s) + U(s)Ms.
\end{equation}
By linearizing the Equation~\eqref{eq: CC} at a central configuration $s_0$,  we get that the Hessian of $U$ restricted on $\E$ at $s_0$ with respect  to the Euclidean product of the ambient space is given by
\begin{equation} \label{eq: Hessian U}
 D^2U|_{\mathcal{E}}(s_0)=D^2U(s_0)+U(s_0)\, M
\end{equation}
whilst wrt the mass inner product,  we get that the Hessian of $U$ at the central configuration $s_0$ can be re-written as follows
\begin{equation}\label{eq: normalized Hessian}
M^{-1}D^2U|_{\mathcal{E}}(s_0)=M^{-1}D^2U(s_0)+U(s_0)\,\Id.
\end{equation}
\begin{defn} \label{def: BS08 condition}
Let $s_0 \in \E$ be a central configuration.\\
 We denote by $\mathfrak s(s_0)=\Set{\lambda_j(s_0)| j=1, \ldots, n^*-1}$ the spectrum of $M^{-1}D^2U|_{\mathcal{E}}(s_0)$ and we assume that
\[
\lmd_1(s_0) \le \lmd_2(s_0) \le \cdots \le \lmd_{n^*-1}(s_0).
\]
We say that  $s_0$ satisfies the condition
\begin{enumerate}
\item {\bf Spiral} if $\lmd_1(s_0)<-\dfrac{1}{8}U(s_0)$
\item {\bf  Non-spiral}  if $\lmd_1(s_0) \ge -\dfrac{1}{8}U(s_0)$
\item {\bf Strict non-spiral} or {\bf [BS]}  if $\lmd_1(s_0) > -\dfrac{1}{8} U(s_0)$.
\end{enumerate}
\end{defn}
\begin{rem}
The spiral condition firstly appeared in the study of the isosceles three body problemby Devaney \cite{Dev80} and Moeckel \cite{Moe81}. The strict non-spiral condition in the variational setting was introduce by authors  in \cite{BS08}. In this last reference they  proved that under the spiral condition, the Morse index of a $s_0$-asymptotic colliding trajectory is infinite. We refer the interested reader to  \cite{BHPT20} where a complete index theory has been constructed  for $s_0$-asymptotic colliding motions.

By starting from the results proved in \cite{BHPT20}, authors in  \cite{HOY21}
constructed a new index theory by giving a complete description of the Morse index in the case of  a doubly asymptotic solutions that are $\mathscr C^2$ solutions of the Newton's equation of motion having convergent angular part, meaning that the following limits $\lim_{t\rightarrow T^{\pm}} s(t) = s_{\pm}$ exist. For this class of solutions, they proved that  for a doubly asymptotic homothetic solution $\phi$, the following alternative occurs:
\begin{itemize}
	\item If the {\sc non-spiral condition} holds and the 
	\begin{itemize} 
	\item {\sc  Energy is negative}, then the Morse index of $\phi$ is  equal to the Morse index of the corresponding normalized central configuration $s_0$
	\item {\sc Energy is non-negative}, then  the Morse index of $\phi$ is zero
	\end{itemize}
	\item If the {\sc spiral condition} holds, then the  Morse index  of $\phi$ is infinite.
\end{itemize}
\end{rem}
The aim of this paper is to investigate a special class of motions provided by the $s_0$-colliding  homothetic motions. A special property for these class of motions is that  their angular part doesn't depend on the time $t$ and it is constantly equal to $s_0$.

In this paper we prove  that in the non-spiral case, irrespective of how large the index of the limiting central configuration $s_0$  and how large the energy level of the orbit is, we get that a $s_0$-colliding homothetic motion is always a Morse minimizer in the sense that has a vanishing   Morse index. We observe also  that, for a $s_0$-asymptotic motion and in the equal case for the non-spiral condition, the second variation is not anymore a Fredholm quadratic form even if  the Morse index given at Definition~\ref{dfn: Morse index}, is still finite. The main result of this paper reads as follows. 
\begin{mainthm}\label{thm:2homothetics}
Let $s_0 \in \E$ be a central configuration and let $q$ be a  colliding $s_0$-homothetic $\mathscr C^2$-solution $q \in \mathscr C^2([0, T^+), \Xh)$. If $s_0$  satisfies the
\begin{enumerate}
\item[(a)] {\sc Spiral condition}, then $\morse (q; 0, T^+)= +\infty$
\item[(b)] {\sc Non-spiral condition}, then $\morse (q; 0, T^+)= 0$.
\end{enumerate}
\end{mainthm}


\subsection*{Notation}.
\small{
The following notation will be used without further comments throughout the paper.
\begin{itemize}
\item  For any positive integer $k$, we denote by $(\R^{2k}, \omega )$ the standard $(2k)$-dimensional symplectic space
\item $\Id$ denotes the identity matrix and $J= \left( \begin{array}{cc}
0 & -\Id \\ \Id & 0
\end{array} \right)$
the standard complex structure
\item Given  $f \in \mathscr C^2(\R^k, \R)$, $\nabla f$ will denote the usual  Euclidean gradient  whilst $D^2 f$ the Euclidean Hessian of $f$
\item  We denote by  $\ld $ (resp. $\ln$) the {\bf Dirichlet} (resp. {\bf Neumann}) Lagrangian subspace of the symplectic space $(\R^{2k}, \omega)$
\item For a given  finite set $A$, we denote by  $\# A$ its cardinality
\end{itemize}
}

The paper is organized as follows:
\tableofcontents


\section{The intersection  index of a homothetic motion}
\label{sec:McGehee-coordinates}

The aim of this section is to associate to an $s_0$-homothetic motion a  {\bf geometrical index} defined as an intersection index between a path of Lagrangian subspaces and a cooriented subvariety of the Lagrangian Grassmannian manifold. It turns out that this geometrical index coincides up to a constant (depending on the dimension of the configuration space) with the Morse index of the motion. This equality is a direct consequence of a {\sc Morse-type index theorem.}

A key idea for constructing a suitable index theory in this framework is to equip  the phase space by the {\sc McGehee coordinates}. However the  main difficulty to face with is that  this coordinates transformation is not canonical and so the Hamiltonian structure of the problem has not preserved. In Subsection~\ref{subsec: McGehee coordinates}, we construct an Hamiltonian version of these coordinates. Our basic reference is \cite{HOY21} and references therein.


\subsection{McGehee coordinates.}\label{subsec: McGehee coordinates}

Let $q \in \mathscr C^2([0, T^+), \Xh)$ be an $s^+$-homothetic colliding solution having  energy $h_0$; in particular
\begin{equation} \label{eq: lim s}
\lim_{t \to T^{+}} s(t)= \lim_{t \to T^{+}} q(t)/r(t) : = s^{+} \qquad \textrm{ and }\qquad \lim_{t\rightarrow T^{+}}q(t)=0.
\end{equation}

Let us now introduce on the $n^*:=d(n-1)$-dimensional inertia ellipsoid $\E$ a smooth coordinate chart $(\Omega, \psi)$ defined by
\[
\psi: \Omega \ni s \longmapsto x \in \psi(\Omega) \subset \rr^{n^*-1}
\]
and we set $\U(x):= U(\psi^{-1}(x))$ and
\[
\Mh (x):= \left(\dfrac{\partial \psi^{-1}}{\partial x} \right)^T M
\left(\dfrac{\partial \psi^{-1}}{\partial x} \right).
\]
So, in this new variables, by setting    $x^{+} :=\psi(s^{+})$, then we get
\begin{equation}
\lim_{t \to T^{+}} x(t) = \lim_{t \to T^{+}} \psi(s(t)) =x^{+}
\qquad \textrm{ and } \qquad
  \Mh _{+} := \lim_{t \to T^{+}}  \Mh (x(t)) =  \Mh (x^{+}).
\end{equation}

\begin{rem}
It is worth observing that the matrix $\Mh(x)$ is not anymore constant being, in fact,  dependent on  $x$. However, it is straightforward to check that  $\Mh(x)$ is (block) symmetric and in fact the following equality holds:
\[
\trasp{\Mh}(x)= \Mh(x) \qquad \forall\, x \in  \R^{n^*-1}.
\]
\end{rem}
\begin{note}
In shorthand notation we drop the dependence of the variable $x$ in the matrix $\Mh$.
\end{note}
Denoting by  $\Sigma:=(0, \infty) \times \R^{n^*-1}$ the open cylinder, in these new
coordinates, the  Lagrangian function $L: T\Sigma \to \R$ is given by
\[
L(r,x,\dot{r},\dot{x})  =K(r, x, \dot{r}, \dot{x})+r^{-1}\U(x)
					    =\dfrac{1}{2}\left(\dot{r}^2+r^2 \langle \Mh \xd, \xd \rangle\right) + r^{-1} \U(x).
\]
The next step is to introduce the following change of coordinates in the fibers of $T^* \X $
\[
\begin{cases}
	 p_1=\dot{r}\\
	  p_2=r^2\,\widehat{M}\,\dot{x}.
\end{cases}
\]
\begin{rem}
We observe that the asymmetry in the two components of the linear momentum vector is a direct consequence of the asymmetry of the kinetic term in these new variables.
\end{rem}
In this new coordinates the corresponding Hamiltonian, reads as follows
\begin{equation}
H(p_1, p_2, r, x)= \dfrac{1}{2} \left( p_1^2+\dfrac{\langle  \Mh ^{-1}p_2,p_2\rangle }{r^2} \right)-\dfrac{\U(x)}{r}.
\end{equation}
By setting  $\zt(t)=(p_1,p_2,r,x)(t)$, then we get that Newton's equations of motion in these new coordinates reduces to the following Hamiltonian system
\begin{equation} \label{eq: Hamiltonian-equation-McG}
\dot{\zt}=J \nabla H(\zt),
\end{equation}
where
\[
 \nabla H(\zt)=\left(p_1, \dfrac{ \Mh ^{-1}p_2}{r^2},  \dfrac{\U(x)}{r^2}-\dfrac{\langle  \Mh ^{-1}p_2,p_2\rangle }{r^3}, \dfrac{\nabla_x\langle   \Mh ^{-1} p_2,p_2\rangle }{2r^2}-\dfrac{\nabla_x \U(x)}{r}\right).
 \]
 By linearizing he Equation~\eqref{eq: Hamiltonian-equation-McG} along a solution, then we get the following linear  Hamiltonian system
\begin{equation} \label{eq:linearized-HAM}
\dot{\xi}(t)= JB(t)\, \xi(t) := JD^2H(\zt(t))\xi(t),
 \end{equation}
where
\[
 D^2H(\zt(t))=
\begin{pmatrix}
    1 &   0 & 0 &   0\\
    0   & \dfrac{\Mh ^{-1}}{r^2} &   -\dfrac{2}{r^3}\Mh ^{-1}p_2  & \dfrac{\nabla_x(\Mh ^{-1}p_2)}{r^2}\\
    0 &   -\dfrac{2}{r^3}p_2^T\Mh ^{-1} &  \dfrac{3\langle   \Mh ^{-1}p_2,p_2\rangle }{r^4}-\dfrac{2\U(x)}{r^3} &  \dfrac{\nabla_x\U(x)}{r^2}-\dfrac{\nabla_x\langle   \Mh ^{-1}p_2,p_2\rangle }{r^3} \\
    0   & \dfrac{\nabla_{x}(p_2^T\Mh ^{-1})}{r^2} &    \dfrac{\nabla_x\U(x)}{r^2}-\dfrac{ \nabla_x\langle   \Mh ^{-1}p_2,p_2\rangle }{r^3} & \dfrac{\nabla^2_x\big(\langle   \Mh ^{-1}p_2,p_2\rangle\big)}{2r^2}-\dfrac{\nabla^2_x\U(x)}{r}
  \end{pmatrix}
\]
It is worth noticing that the Hamiltonian system provided at
Equation~\eqref{eq:linearized-HAM} is still singular since  $r(T^{+})=0$. The next
step consists in regularizing the total collision singularity. This will be done
by using a  well-known change of coordinates introduced by  McGehee  in his celebrated paper \cite{McG74}.

For, we define the new coordinates  $v$ and $u$ and we rescale time as follows
\begin{equation}\label{eq:vuMcGehee 1}
\begin{cases}
v=r^{1/2} p_1=r^{1/2}\dot{r}\\[7pt]
u=r^{-1/2}p_2 =r^{3/2}\Mh \dot{x}
\end{cases}\qquad dt = r^{3/2} \, d\tau.
\end{equation}
In this new coordinates system the Equation~\eqref{eq: Hamiltonian-equation-McG} reduces to the following
\begin{equation}\label{eq:McGehee1.1}
\begin{cases}
v'=\dfrac{1}{2}v^2+\langle  \Mh ^{-1}u,u\rangle -\U(x) \\
u'=-\dfrac{1}{2}uv+\U_x(x)-\dfrac{1}{2}\langle (\Mh ^{-1})_xu,u\rangle \\
r'=rv\\
x'=\Mh ^{-1}u
\end{cases}
\end{equation}
where we denoted by $'$ the $\tau$-derivative.
\begin{rem}
From the kinetic viewpoint  the effect of the this time change is related to get a sort of a colliding trajectory in slow-motion.
\end{rem}
We observe that, in these new coordinates we get that Equation~\eqref{eq:vuMcGehee 1} reduces to
\[
\begin{cases}
v=r^{1/2} p_1=r^{1/2}\dot{r}= r' r^{-1}\\[7pt]
u=r^{-1/2}p_2 =r^{3/2}\Mh \dot{x}= \Mh x'
\end{cases}
\]
and the energy relation fits into the following
\begin{equation}\label{eq:energyM1}
\dfrac{\langle  \Mh ^{-1}u,u\rangle +v^2}{2}-\U(x)=rh_0.
\end{equation}
The following asymptotic behavior hold
\begin{lem}\label{lem:masy1}
Let $T^{+}$ be a total collision instant  for the $s^+$-asymptotic motion $q$. Then the following holds:
\begin{enumerate}
\item[(a)] $\tau=\tau(t) \longrightarrow + \infty$ as $t \to T^{+}$
\item[(b)] $(|v|,u)(\tau) \longrightarrow (\sqrt{2U(s^{+})}, 0)$ as $\tau \to + \infty$.
\end{enumerate}
\end{lem}
\begin{proof}
We refer the interested reader to \cite[Lemma 2.2]{HOY21} for the proof. 	
\end{proof}
\begin{rem}
Without loss of generality, in this paper  we always assume that $\tau(0)=0$. Since
$v= r^{1/2} \dot{r}$ with $r>0$, the sign of $v$ coincides with that of $\dot{r}$.
Since we are dealing with  homothetic motions starting from a central configuration
and total collapsing at the final instant, it holds $\dot{r}(t)\le 0$; so, in particular  $v(\tau)\le 0$ for $\tau\in[0, +\infty)$.
\end{rem}
As already observed before, the disadvantage of this change of coordinates is that the first order system provided at  Equation~\eqref{eq:McGehee1.1} is not Hamiltonian, anymore. The first striking consequence of this fact is that  the index theory cannot be developed in this setting!

In order to overcome this difficulty we consider the first order system given at Equation~\eqref{eq:McGehee1.1}  can be rewritten as follows
\begin{equation}
\xi'(\tau)=J B(\tau) \xi(\tau):=r^{3/2}(\tau)JD^2 H(\zt(\tau))\xi(\tau)\label{eq: linearied tau}
\end{equation}
where
\begin{equation}
 B(\tau)=
 \begin{pmatrix}
    r^{3/2} &   0 & 0 &   0\\
    0   & \dfrac{\Mh ^{-1}}{\sqrt{r}} &   -\dfrac{2\Mh ^{-1}p_2}{ (\sqrt{r})^3} & \dfrac{\nabla_x(\Mh ^{-1}p_2)}{\sqrt{r}}\\
    0 &   -\dfrac{2\trasp{p_2}\Mh ^{-1}}{(\sqrt{r})^3} &  \dfrac{3\langle   \Mh ^{-1}p_2,p_2\rangle }{(\sqrt{r})^5}-\dfrac{2\U(x)}{(\sqrt{r})^3} &  \dfrac{\nabla \U(x)}{\sqrt{r}}- \dfrac{\langle  \nabla_x \Mh ^{-1}p_2,p_2\rangle}{(\sqrt{r})^3} \\
    0   & \dfrac{\nabla_x(\trasp{p_2}\Mh ^{-1})}{\sqrt{r}} &    \dfrac{\nabla_x\U(x)}{\sqrt{r}}- \dfrac{\nabla_x\langle   \Mh ^{-1}p_2,p_2\rangle }{(\sqrt{r})^3} & \dfrac{\nabla^2_x\langle   \Mh ^{-1}p_2,p_2\rangle}{2\sqrt{r}}-r^{1/2}\nabla^2_x\U(x)
  \end{pmatrix}
\end{equation}


\subsection{The geometrical index for a homothetic motion}

We are now entitled to introduce the  the geometrical index  of a homothetic motion and to relate it to the Morse index of a homothetic solution in the sense of Definition~\ref{dfn: Morse index}.  We start with this useful preliminary result  whose proof is straightforward.
\begin{lem}\label{thm:change-R}
Let $\tau \mapsto R(\tau)$ be a path of symplectic matrices.
If $\xi'(\tau)=J B(\tau) \xi(\tau)$, then the path $\tau \mapsto \eta(\tau)$ pointwise defined by $\eta(\tau)=R(\tau)\,\xi(\tau)$ satisfies $\eta'(\tau) =J\,B_R(\tau)\,\eta(\tau)$ with
\begin{equation}\label{eq:B_R}
B_R(\tau):=-JR'(\tau)R^{-1}(\tau)+ \itrasp{R}(\tau)B(\tau)R^{-1}(\tau).
\end{equation}
In particular, if $R(\tau)$ is a constant matrix, then we get that
\[
B_R(\tau)=\itrasp{R} B(\tau)R^{-1}.
\]
\end{lem}
\begin{proof}
The proof is straightforward and it is left to the reader.
\end{proof}
\begin{rem}
It is worth observing that  the matrix path $\tau\mapsto B_R(\tau)$ is symmetric  once observed that the first term in the (RHS) of Equation~\eqref{eq:B_R} is symmetric (since the second term it is). In fact,  the matrix path $\tau \mapsto R'(\tau)R^{-1}(\tau)$ is Hamiltonian being the derivative of the logarithmic of $\tau \mapsto R(\tau)$. In particular $X_R(\tau)=J\, B_R(\tau)$ defines a linear  Hamiltonian vector field.
\end{rem}
\begin{note}
In  shorthand notation, we set $\widehat B (\tau):=B_R(\tau)$ where $B_R(\tau)$ given by \eqref {eq:B_R}.
\end{note}
We set
\begin{equation} \label{eq: R I}
R(\tau) =\diag(r^{3/4},r^{-1/4}\Id,r^{-3/4},r^{1/4}\Id)(\tau)
\end{equation}
and  by Lemma ~\ref{thm:change-R}  it follows that the function $\eta$ pointwise defined by  $\eta(\tau)=R(\tau)\xi(\tau)$ satisfies the Hamiltonian system
\begin{equation}\label{eq:Ham-important}
\eta'(\tau) =J\bh(\tau)\, \eta(\tau),
\end{equation}
where
\begin{multline}\label{eq:hatBtau}
\bh(\tau)= \\
\begin{pmatrix}
    1 &   0 & -\dfrac{3}{4}v &   0\\
    0   &  \Mh ^{-1} &   -2  \Mh ^{-1}u & D_x(\Mh ^{-1}u)+\dfrac{vI}{4}\\
    -\dfrac{3}{4}v &   -2 u^T\Mh ^{-1}&  3\langle   \Mh ^{-1}u,u\rangle -2\U(x) &  \nabla_x\U(x)-\nabla_x\langle   \Mh ^{-1}u,u\rangle  \\
    0   & \nabla_x(u^T\Mh ^{-1})+\dfrac{vI}{4} &   \nabla_x\U(x)-\nabla_x \langle   \Mh ^{-1}u,u\rangle  & \dfrac{1}{2}\nabla^2_x\langle   \Mh ^{-1}u,u\rangle -\nabla^2_x\U(x)
  \end{pmatrix}
\end{multline}

\begin{lem}\label{lem:eigenvalueUhat}
Let $s_0$ be a (normalized) central configuration and we let $x_0= \psi(s_0)$. Then the following relations between the spectra of the Hessians hold:
\begin{multline}
\mathfrak s(s_0):=\mathfrak s\left(M^{-1}D^2U|_{\mathcal{E}}(s_0)\right)=
\mathfrak s\left(\trasp{A} D^2\U(x_0)A\right)\quad \textrm{ where} \\
\trasp{A} \Mh_0 A=\Id \qquad \textrm{ and } \qquad \\
  \Mh_0 = \left. \left( \dfrac{\partial \psi^{-1}}{\partial x} \right) \right|_{x= x_0}^T M \left. \left( \dfrac{\partial \psi^{-1}}{\partial x} \right) \right|_{x= x_0}.
  \end{multline}
\end{lem}
\begin{proof}
	For the proof, we refer the interested reader to \cite[Lemma 2.1]{HOY21}.
\end{proof}
Let $q \in \mathscr C^2([0, T^+), \Xh)$ be a   homothetic  colliding solution starting  from a central configuration $s^+$ and let $\gamma$ be the fundamental solution of the linearizing  Hamiltonian system along $q$ given at Equation~\eqref{eq:linearized-HAM}. So, $\gamma$ is the matrix solution of the following IVP
\[
\begin{cases}
\dot{\gamma}(t)=JB(t)\gamma(t) \qquad t \in [0,T^+)\\
\gamma(0)=\Id.
\end{cases}
\]
For every $T \in [0, T^+)$, the Morse index of the restriction $q|_{[0,T]}$ on $W^{1,2}_0([0,T], \Xh)$ is well-defined. Now, denoting by $\mu(\ld , \gamma(t) \ld ; [0,T])$ the (intersection) {\bf Maslov index} of the Lagrangian path $t\mapsto \gamma(t)\, \ld$ with the Maslov cycle having vertex at $\ld$ (cfr. Appendix~\ref{sec:app-Maslov} for the basic definitions and properties of the Maslov index), the following {\sc Morse-type index theorem} holds.
\begin{thm}\label{thm:Mor-Mas}
Under the above notation the following equality holds:
\begin{equation}
\morse(q; 0, T^+)+n^*=\mu(\ld , \gamma(t) \ld ; [0,T])
\end{equation}
where the right hand side is the Maslov index of the  Lagrangian path $t \mapsto \gamma(t)\ld $ with the fixed Lagrangian subspace $\ld$.
\end{thm}
\begin{proof}
We refer the interested reader to the paper \cite{HS09} for the proof. 	
\end{proof}
We are now in position to introduce the {\sc geometrical index of a  homothetic solution.}
\begin{defn}\label{def:geo-index}
	Let $s^+$ be a (normalized) central configuration and let $q \in \mathscr C^2([0, T^+), \Xh)$ be a   homothetic  solution having total collapse at the final instant $T^+$. We term {\bf geometrical index} of $q$ the integer defined by
	\[
	\igeo(q):=  \mu(\ld , \widehat{\gamma}(\tau)\ld ; [0,\mathcal{T}])
	\]
	where  $\tau(0)=0$ and $\tau(T)=\mathcal{T}$ and where $\widehat{\gamma}$ denotes the fundamental solution of the Hamiltonian system given at Equation~\eqref{eq:Ham-important}.
\end{defn}

In order to compute the limit of the above Morse index as $T$ converge to $T^{+}$, the results obtained in the previous subsection using McGehee coordinates will be crucial. Recall that the change of coordinates transfers  \eqref{eq:linearized-HAM} to the following system
\begin{equation} \label{eq: linear Ham tau}
\eta'(\tau) = J \bh(\tau) \eta(\tau), \; \text{ where } \; \bh(\tau) = B_R(\tau).
\end{equation}
Recall that $B_R(\tau)$ is defined in \eqref {eq:B_R} with $R(\tau)$ given by \eqref{eq: R I} in McGehee coordinates.
The fundamental solution of \eqref{eq: linear Ham tau} will be denoted as $\widehat{\gm}(\tau)$, where $\widehat{\gm}(0)=I$.

Next lemma shows the Maslov index is invariant under the change of the McGehee coordinates.
We are entitled to

\begin{lem} \label{lem:Maslovindex-t-tau} For any $0\le T < T^+$, in McGehee coordinates,
$$ \mu(\ld , \widehat{\gamma}(\tau)\ld ; [0,\mathcal{T}]) =\mu(\ld , \gm(t)\ld ; [0, T]), $$

\end{lem}
\begin{proof}
	We refer the interested reader to \cite[Lemma 3.1]{HOY21}.
\end{proof}


\subsection{Asymptotic behavior of the geometrical index}

Given  the homothetic $q \in \mathscr C^2([0, T^+), \Xh)$ starting from a central
configuration $s_0$ and collapsing to the total collision, we get that
$x(\tau) \equiv x_0 = \psi(s_0)$ is constant. In particular we get that the
variable $u$ vanishes identically for all $\tau$. A direct consequence of this is that first order system given at Equation~\eqref{eq:McGehee1.1} reduces to the following
\begin{equation}\label{eq:McGehee2}
\begin{cases}
v'&=\dfrac{1}{2}v^2-b, \\
r'&=rv
\end{cases}
\end{equation}
where $b = \U(x_0)$. In fact the second and fourth equations at Equation~\eqref{eq:McGehee1.1} vanishes identically. Moreover,  the energy identity given at Equation~\eqref{eq:energyM1} reads as follows
\begin{equation}\label{eq:energy-relation}
\dfrac{1}{2}v^2-b=rh_0 .
\end{equation}
Moreover, the matrix given at Equation~\eqref {eq:hatBtau} reduces to:
\begin{multline}
\bh(\tau)=
\begin{pmatrix}
    1 &   0 & -\dfrac{3}{4}v &   0\\
    0   &  \Mh ^{-1} &   0 &\dfrac{vI}{4}\\
    -\dfrac{3}{4}v &   0&  -2b & 0  \\
    0   & \dfrac{v\Id}{4} &   0   & -\U_{xx}(x)
  \end{pmatrix}= \widehat B _1(\tau)\diamond \widehat B _2(\tau) \qquad \textrm{ for }\\[12pt]
  \widehat B _1(\tau)=
  \begin{pmatrix}
1 & -\dfrac{3}{4}v\\
-\dfrac{3}{4}v&-2b
\end{pmatrix}\qquad \textrm{ and }   \qquad
\widehat B _2(\tau)=
\begin{pmatrix}
\Mh ^{-1} & \dfrac{1}{4}vI\\
\dfrac{1}{4}vI&-\nabla^2_x\U(x_0)
\end{pmatrix}
\end{multline}
where  $\Mh =  \Mh (x_0)$.\footnote{Here $\diamond$ represents the {\bf symplectic sum} introduced by Long (see \cite{Lon4}) and defined for any two $2m_k\times 2m_k$ matrices  $O_k=\left(\begin{array}{cc}A_k&B_k\\
                             C_k&D_k\end{array}\right)$ for  $k=1, 2$, as follows:
\begin{equation} O_1 \diamond O_2=\left(
  \begin{array}{cccc}
   A_1 &   0 & B_1 &   0\\
                            0   & A_2 &   0 & B_2\\
                           C_1 &   0 & D_1 &   0\\
                           0   & C_2 &   0 & D_2  \\
  \end{array}
\right).
\end{equation}
The relevance of this product is that the Maslov index is additive wrt this product.
}
By using the symplectic additivity of the Maslov index, we get
\begin{equation}\label{Symadd}
\mu(\ld , \widehat{\gamma}(\tau) \ld ; [0,\mathcal{T}])=\mu(\ld , \widehat{\gamma}_{1}(\tau) \ld ; [0, \mathcal{T}])+\mu(\ld , \widehat{\gamma}_{2}(\tau) \ld ; [0,\mathcal{T}]),
\end{equation}
where $\widehat \gamma_1$ and $\widehat \gamma_2$ are respectively  solutions of the following two Hamiltonian  systems
\begin{equation}\label{eq:dueHamsys}
\begin{cases}
\gamma'(\tau)=J\widehat B _{1}(\tau)\gamma(\tau) \\[7pt]
\gamma(0)=\Id.
\end{cases} \qquad \textrm{ and } \qquad
\begin{cases}
\gamma'(\tau)=J\widehat B _{2}(\tau)\gamma(\tau)\\[7pt]
\gamma(0)=\Id.
\end{cases}
\end{equation}
The next result compute the  Maslov index given by the Lagrangian path defined through the phase flow of the first Hamiltonian system given at Equation~\eqref{eq:dueHamsys} on the whole right halfline.
\begin{lem}\label{prop:B1}
Under the above notation, the following result holds
\begin{equation}
\lim_{\mathcal{T}\rightarrow+\infty}\mu(\ld , \widehat{\gamma}_{1}(\tau)\ld ;[0,\mathcal{T}])=1,
\end{equation}
\end{lem}
\begin{proof}
We let $\widetilde{\gamma}_{1}(\tau)=R_{1}(\tau)\,\widehat{\gamma}_{1}(\tau)\, R^{-1}_{1}(0)$, then
\begin{equation}
\begin{cases}
\widetilde \gamma'_{1}(\tau)=J\,B_{R_1}(\tau)\,\widetilde \gamma'_{1}(\tau)\\[7pt]
\widetilde \gamma'_{1}(0)=I.
\end{cases}
\end{equation}
By invoking Lemma~\ref{lem:Maslovindex-t-tau}, we immediately get that
\begin{equation}
\mu(\ld ,\widetilde \gamma_{1}(\tau)\ld ;[0,\mathcal{T}])=\mu(\ld , \widehat{\gamma}_{1}(\tau)\ld ;[0,\mathcal{T}]).\nonumber
\end{equation}
We let $\widetilde \gamma_{1}(\tau)=
\begin{pmatrix}
a(\tau) & h(\tau) \\
c(\tau) & k(\tau) \\
\end{pmatrix}
$.
Then since the curve of Lagrangian subspaces $\tau \mapsto \widetilde \gamma_{1}(\tau)\, L_D$ is a plus curve with respect  to $L_D$, the following equality holds
\begin{equation}
\mu(\ld , \widetilde \gamma_{1}(\tau)\ld ;[0,\mathcal{T}])=\sharp\Set{\tau \in [0, \mathcal{T}]|c(\tau)=0}
\end{equation}
where $c$ is a solution of the following IVP
\begin{equation}
\begin{cases}
c''(\tau)=\left[\dfrac{3}{16}v^{2}+\dfrac{11}{4}b\right]\,c(\tau) \\[5pt]
c^{\prime}(0)=1\\[5pt]
c(0)=0.
\end{cases}
\end{equation}
Since $\dfrac{3}{16}v^{2}+\dfrac{11}{4}b>0$ then we get that  $c$  vanishes only at the starting point $[0, \mathcal{T}]$ and so we get that $
\sharp\Set{\tau\in[0,\mathcal{T}]| c(\tau)=0}=1$. By taking into account Equation~\eqref{eq:CLM-index} and by observing that the contribution provided by the starting instant to the Maslov index is given by the coindex of the crossing form, then we get
\begin{equation}
\mu(\ld ,\widehat{\gamma}_{1}(\tau)\ld ;[0,\mathcal{T}])=1.
\end{equation}
The conclusion follows by observing that for ever $\mathcal{T}$ there are no further crossings on $[0, \mathcal{T}]$.
\end{proof}
Now, let $A$ be the matrix defined at Lemma~\ref{lem:eigenvalueUhat} and let $A_d$ the block diagonal matrix defined by  $A_d=\diag(A^{-T},A)$. By a straightforward calculation, we get
\[
\widehat B_{2, A_d}(\tau)= \trasp{A_d}\widehat B_{2}(\tau) A_d =
 \begin{pmatrix}
\Id & \dfrac{1}{4}v\Id\\
\dfrac{1}{4}v\Id&-\trasp{A}\nabla^2_x\U(x_0)A
\end{pmatrix}
\]
For $i=1, \dots, n^*-1$, we denote by  $\lambda_i$  the eigenvalues of
$M^{-1}D^2U_{\mathcal{E}}(s_0)$ and  by Lemma~\ref{lem:eigenvalueUhat}
it follows that  they are eigenvalues of  $\trasp{A}\U_{xx}(x_0) A$ as well. By choosing a diagonalizing basis for $\U_{xx}(x_0)$, then it is possible to decompose the matrix $\widehat B_{2, A_d}(\tau)$ as the symplectic sum of $n^*-1$ matrices; so, we get
\[
\widehat B_{2, A_d}(\tau)=\widetilde{B}_{\lambda_1}(\tau)\diamond\cdots\diamond\widetilde{B}_{\lambda_{n^{*}-1}}(\tau)
\]
where
$$
\widetilde{B}_{\lambda}(\tau)=  \left(\begin{array}{cc}1 & \dfrac{1}{4}v\\
                             \dfrac{1}{4}v&-\lambda\end{array}\right),
$$
here to simplify notation, let $\lmd$ represent any $\lmd_i$, $i =1, \dots, n^*-1$.
By the  symplectic additivity of the Maslov index,  we get
\begin{equation}
\label{eq:MasSumB2}  \mu(\ld ,\widehat{\gamma}_{2}(\tau)\ld ;[0,\mathcal{T}]) = \sum_{i =1}^{n^*-1} \mu(\ld ,\widetilde{\gamma}_{\lambda_i}(\tau)\ld ;[0,\mathcal{T}])
\end{equation}
where $\widetilde{\gamma}_{\lambda_i}$ is the matrix solution of
\begin{equation}
\begin{cases}
\gamma'(\tau)=J\widetilde{B}_{\lambda_i}(\tau)\gamma(\tau)
 \\[7pt]
\gamma(0)=\Id.
\end{cases}
\end{equation}

\begin{lem}\label{prop:Neu}
We assume that the strictly non-spiral condition is fulfilled. Then the following holds:
\begin{equation}
\lim_{\mathcal{T}\rightarrow+\infty}\mu(\ld , \widetilde{\gamma}_{\lambda}(\tau)\ld ;[0,\mathcal{T}])=1
\end{equation}
where we set $\lambda=\lambda_i$.
\end{lem}
\begin{proof}
Wlog, we assume that for $t(0)=0$ and we  let
$\widetilde{\gamma}_{\lambda}(\tau)=
\begin{pmatrix}
a(\tau) & h(\tau) \\
c(\tau) & k(\tau)
\end{pmatrix}$,
where  $\tau \mapsto\trasp{(a(\tau), c(\tau))}$ is the solution of the following IVP:
\begin{equation}\label{eq:ode1}
\begin{cases}
a'(\tau)=-\dfrac{1}{4}v(\tau)a(\tau)+\lambda c(\tau)\\
c'(\tau)=a(\tau)+\dfrac{1}{4}v(\tau)c(\tau)\\
a(0)=1\\
c(0)=0.
\end{cases}
\end{equation}
By a direct calculation and by taking into account Equation~\eqref{eq:McGehee2}, we get
\begin{equation}\label{equ:ode2}
\begin{cases}
a''(\tau)=\left[-\dfrac{v(\tau)^{2}}{16}+\dfrac{b}{4}+\lambda\right]a(\tau)\\[7pt]
a'(0)=-\dfrac{1}{4}v(0)\\[7pt]
a(0)=1
\end{cases} \qquad \textrm{ and } \qquad
\begin{cases}
c''(\tau)=\left[\dfrac{3}{16}v(\tau)^{2}-\dfrac{b}{4}+\lambda\right]c(\tau)\\[7pt]
c'(0)=1\\[7pt]
c(0)=0.
\end{cases}
\end{equation}
By using the  energy relation given at Equation~\eqref{eq:energy-relation}, we get
\begin{equation}\label{equ:ode3}
\begin{cases}
a''(\tau)=\left[-\dfrac{rh_{0}}{8}+\dfrac{b}{8}+\lambda\right]a(\tau)\\[7pt]
a'(0)=-\dfrac{1}{4}v(0)\\[7pt]
a(0)=1
\end{cases}
\qquad \textrm{ and } \qquad
\begin{cases}
c''(\tau)=\left[\dfrac{3rh_{0}}{8}+\dfrac{b}{8}+\lambda\right]c(\tau)\\[7pt]
c'(0)=1\\[7pt]
c(0)=0.
\end{cases}
\end{equation}
According to the sign of the the energy level $h_0$, we distinguish two cases:
\begin{itemize}
\item[(i)] {\bf Negative energy.} In this case, since $q \in \mathscr C^2([0, T^+), \Xh)$  collapse at the final instant hence $\tau\mapsto r(\tau)$ is a decreasing function; so, we get $r'(0)\le 0$ and by taking into account of Equation~\eqref{eq:McGehee2}, we get $v(0)\le 0$.  This, in particular,  implies that $a'(0)\ge 0$ in the IVP given at Equation~\eqref{equ:ode3}.

Under the strict non-spiraling condition, we get that  $-\dfrac{rh_{0}}{8}+\dfrac{b}{8}+\lambda>0$. By this we immediately conclude that  $a(\tau)>0$ for every $\tau\in[0, +\infty)$.  On the other hand, always from Equation~\eqref{equ:ode3}, we know that $c(0)=0$ and $c'(0)>0$.  We claim that the function $\tau\mapsto c(\tau)$  has  only one zero in $[0,+\infty)$.

Arguing by contradiction,  we assume that $\tau_{1}>0$ is the first zero of $c(\tau)$ in $(0,+\infty)$.
Then $c(\tau)>0$ in $(0, \tau_{1})$.  and  in particular $c'(\tau_{1})\le 0$. Now, by using   Equation~\eqref{eq:ode1}, we get that  $a(\tau_{1})\le 0$ which is impossible since  $a(\tau)>0$ for every  $\tau\in[0,+\infty)$
\item[(ii)] {\bf  Non-negative energy}. In this case, under the strict non-spiraling condition
\[
\dfrac{3rh_{0}}{8}+\dfrac{b}{8}+\lambda>0.
\]
This in particular  implies that $\tau=0$ is the only zero  of $c(\tau)$ on the interval $[0,+\infty)$.
\end{itemize}
Summing up the previous arguments, we get that under the strict non-spiral condition the only contribution to the Maslov index is given by the starting point. In conclusion, we get
\begin{equation}
\lim_{\mathcal{T}\rightarrow+\infty}\mu(\ld , \widetilde{\gamma}_{\lambda}(\tau)\ld ;[0,\mathcal{T}])=\sharp\{\tau: c(\tau)=0,\tau\in[0,+\infty)\}=1.
\end{equation}
This concludes the proof.
\end{proof}


\section{Proof of Theorem~\ref{thm:2homothetics}}

Let  $q \in \mathscr C^2([0, T^+), \Xh)$ a homothetic orbit colliding at the central configuration $s_0$ at the instant $T^+$ and we assume only the  non-spiral condition holds.\footnote{
It is worth observing that since we are not assuming the strict non-spiral condition, a perturbation argument is needed. }
Given a sufficiently small $\vep >0$, we set
\begin{equation}
\widehat B (\tau,\varepsilon)=\widehat B _{1}(\tau)\diamond\widehat B _{2}(\tau,\varepsilon) \qquad  \textrm{ where } \qquad
\widehat B _{2}(\tau,\varepsilon)=
\begin{pmatrix}
\Mh ^{-1} & \dfrac{1}{4}vI\\
\dfrac{1}{4}vI&-\nabla^2_x\U(x_0)-\varepsilon\Mh.
\end{pmatrix}
\end{equation}
Since by construction $\widehat B (\tau)-\widehat B (\tau,\varepsilon) \ge 0 $ and by using the monotone property of Maslov index given at Equation~\eqref{eq:monotone-property}, we get
\begin{equation}
\mu(\ld ,\widehat{\ga}(\tau)\ld ;[0,\mathcal{T}])\ge
\mu(\ld ,\widehat{\ga}_{\vep}(\tau)\ld ;[0,\mathcal{T}]),\label{96}
\end{equation}
where $\widehat{\ga}_{\vep}(\tau,\tau_{1})$ is the solution of
\begin{equation*}
\begin{cases}
\gamma'(\tau)  =J\widehat B (\tau,\varepsilon)\gamma(\tau) \\
\gamma(0)  =\Id.
\end{cases}
\end{equation*}
We observe that this perturbed system  satisfies the strict non-spiral condition too
and by using Lemma~\ref{prop:B1}, Lemma~\ref{prop:Neu} and the symplectic
decomposition formula provided at Equation~\eqref{Symadd} and Equation~\eqref{eq:MasSumB2}, then we get
\begin{equation}\label{97}
\lim_{\mathcal{T}\rightarrow+\infty}
\mu(\ld ,\widehat{\ga}_{\vep}(\tau)\ld ;[0,\mathcal{T}])=n^{*}.
\end{equation}
We conclude the proof by a contradiction argument. We assume that for $\mathcal{T}$ large enough
\begin{equation}
\mu(\ld ,\widehat{\ga}(\tau)\ld ;[0, \mathcal{T}])\ge n^{*}+1. \nonumber
\end{equation}
By  Theorem~\ref{thm:Mor-Mas}, we get that the $\mu(\ld ,\widehat{\ga}(\tau)\ld ;[0,\mathcal{T}])$
is equal to the Morse index of the solution $q$ which is lower semi-continuous; in  particular doesn't decrease by a small perturbations. So, by choosing  $\varepsilon$ sufficiently small, we get
\begin{equation}
\mu(\ld ,\widehat{\ga}_\varepsilon(\tau)\ld ;[0,\mathcal{T}])\ge n^{*}+1
\end{equation}
which contradicts Equation~\eqref{97}.  So, we get
\begin{equation}
\mu(\ld ,\widehat{\ga}(\tau)\ld ;[0,\mathcal{T}])=n^{*}.
\end{equation}
By these arguments and by using Lemma~\ref{lem:Maslovindex-t-tau} and  Theorem~\ref{thm:Mor-Mas}, we finally get that under the non-spiral condition and for $\mathcal{T}$ sufficiently large
\begin{equation}
m^{-}(q; 0, T^{+})=0.
\end{equation}
This conclude the proof of (b). The proof of the spiral case (a) follows by \cite[Theorem 1.2]{HOY21}. \qed


\appendix

\section{A quick recap of the Maslov index}\label{sec:app-Maslov}

In this section, we briefly recall the definition and the basic properties of the Maslov index. Our basic references are \cite{CLM94} and \cite{RS93} and references therein.

In the {\sc standard symplectic space} $(\R^{2k},\omega)$ a {\sc Lagrangian subspace}
is a linear subspace $L$ such that $L=L^\omega$ where $L^\omega$ denotes the symplectic orthogonal. We denote by $\Lag(k)$ the set of all Lagrangian subspaces and we recall that $\Lag(k)$ is a  compact and connected $k(k+1)/2$-dimensional smooth submanifold   of the Grassmannian of the $k$-dimensional subspaces of $\R^{2k}$.


\subsection{An axiomatic definition}

Given  two continuous paths of Lagrangian subspaces $t\mapsto L_1(t)$ and $t \mapsto L_2(t)$ with  $t\in[a,b]$, it is possible to define a integer-valued homotopy invariant known as {\bf  Maslov index for Lagrangian paths} and denoted by $\mu(L_1,L_2; [a, b])$. This  invariant is uniquely determined by a set of axioms listed as Property I to VII below. (For further details, we refer the interested reader to   \cite{CLM94} and references therein).
\begin{itemize}
\item {\sc Property I. (Reparametrization invariance)}  Let $\vr:[c,d]\rightarrow [a,b]$ be a continuous and piecewise smooth function satisfying $\vr(c)=a$, $\vr(d)=b$, then
 \begin{equation}
\mu(L_1(t), L_2(t); [a,b])=\mu(L_1(\vr(\tau)), L_2(\vr(\tau));[c,d]).
\end{equation}

\item {\sc Property II. (Homotopy invariant with end points)} If two continuous
families of Lagrangian paths $(s,t) \mapsto L_1(s,t)$ and $(s,t)\mapsto L_2(s,t)$,
with $(s,t)\in [0,1] \times [a,b]$ satisfy the following property
\[
\dim\big(L_1(s,a)\cap L_2(s,a)\big)=C_1 \qquad \textrm{ and } \qquad \dim(L_1(s,b)\cap L_2(s,b)) = C_2\qquad \forall \, s \in [0,1]
\]
where $C_1, C_2$ are two positive constants,  then
\begin{equation}
\mu(L_1(0,t), L_2(0,t))=\mu(L_1(1,t),L_2(1,t)).
\end{equation}

\item {\sc Property III. (Path additivity)}  If $a<c<b$, then
\begin{equation}
\mu(L_1(t),L_2(t))=\mu(L_1(t),L_2(t); [a,c])+\mu(L_1(t),L_2(t); [c,b]).
\label{adp1.3}
\end{equation}

\item {\sc Property IV. (Symplectic invariance)} Let $\gamma \in \mathscr C^0([a,b], \Sp(2n))$. Then
\begin{equation}
\mu(L_1(t),L_2(t); [a,b])=\mu(\ga(t)L_1(t), \ga(t)L_2(t); [a,b]).
\end{equation}

\item {\sc Property V. (Symplectic additivity)} For $i=1,2$ let $(W_i, \omega_i)$ be  two symplectic spaces and we assume that  $L_i \in \mathscr C^0([a,b], \Lag(W_1))$ and $L_i \in \mathscr C^0([a,b], \Lag(W_2))$ for $i =1, 2$.  Then
\begin{equation}
\mu(L_1(t)\oplus \widehat{L}_1(t),L_2(t)\oplus \widehat{L}_2(t); [a,b])= \mu(L_1(t),L_2(t))+
\mu(\widehat{L}_1(t),\widehat{L}_2(t); [a,b]).
\end{equation}

\item {\sc Property VI. (Symmetry)} For $i=1,2$ we let  $L_i \in \mathscr C^0([a,b],\Lag(k))$. Then
\begin{equation}  \mu(L_1(t), L_2(t); [a,b])= \dim L_1(a)\cap L_2(a)-\dim L_1(b)\cap L_2(b)  -\mu(L_2(t),L_1(t)).
\end{equation}

\item {\sc Property VII (Monotone property)}   For $j=1,2$ we let   $L_j(t)=\gamma_j(t)\,V$, where
\[
\begin{cases}
\dot \gamma_j(t)=JB_j(t)\gamma_j(t)\\
\ga_j(t)=\Id_{2k}.
\end{cases}
\]
If $B_1(t)\ge B_2(t)$ meaning that the quadratic form associated to $B_1(t)-B_2(t)$ is positive definite, then for any $V_0,V_1\in Lag(2n)$, we get
\begin{equation}\label{eq:monotone-property}
 \mu(V_0, \gamma_1V_1;[a,b])\ge  \mu(V_0, \gamma_2V_1;[a,b]).
       \end{equation}
\end{itemize}


\subsection{An efficient way to compute the Maslov index}

An efficient way to compute  the Maslov index is via the crossing forms introduced by authors in \cite{RS93}.  For simplicity and since it is  enough for our purpose, we only briefly recall this procedure for the computation of the Maslov index for a path of Lagrangian subspaces with respect to a fixed Lagrangian subspace.

Let $\Lambda\in \mathscr C^1([a,b], \Lag(k))$ be a
$\mathscr C^1$-curve of Lagrangian subspaces with $\Lambda(0)=\Lambda$ and
let $V$ be a fixed Lagrangian subspace which is transversal to
$\Lambda$. For $v\in \Lambda$ and for small $t$, we define $w(t)\in V$ by requiring that $v+w(t)\in \Lambda(t)$. Then it is easy to check that
the  quadratic form  $Q$ defined by
\begin{equation}
Q(v)=\left.\dfrac{d}{dt}\right|_{t=0}\omega (v,w(t))
 \end{equation} is
independent of the choice of $V$ (cfr.\cite{RS93} for further details). Now, given a Lagrangian subspace $W$, we say that $t$ is a {\bf crossing instant} of the path $t\mapsto \Lambda(t)$ with $W$ if $\dim\big(\Lambda(t)\cap W) \ge 1$.   At each crossing, we define  the {\bf crossing form} as follows
 \begin{equation}
 \Gamma(\Lambda(t),W,t)=Q|_{\Lambda(t)\cap W}
 \end{equation}
and we say that a crossing is  {\bf  regular} if the crossing
form is non-degenerate.
\begin{rem}
It is worth observing that in the application, usually the Lagrangian path $t\mapsto \Lambda(t)$ is pointwise defined by $\Lambda(t)= \gamma(t)\, L$ where $t \mapsto \gamma(t)$ is a path of symplectic matrices corresponding to the fundamental solution of a Hamiltonian system. In this case the crossing form can be explicitly written in terms of the entries of the path $\gamma$. In fact, if the path is given by $\Lambda(t)=\gamma(t)\,L$ with $\gamma(t)\in \Sp(2k)$ and
$L\in Lag(k)$, then the crossing form is equal to
\begin{equation}\label{eq:utile}
\langle -\gamma(t)^TJ\dot{\gamma}(t)v,v\rangle \qquad \textrm{ for  } v\in
\gamma(t)^{-1}(\Lambda(t)\cap W)
\end{equation}
where $\langle\cdot, \cdot\rangle $ denotes the standard inner product of
$\mathbb{R}^{2k}$. It is immediate to check that the monotone property of the Maslov index given at Equation~\eqref{eq:monotone-property} is a direct consequence of Equation~\eqref{eq:utile}. 
\end{rem}
Assuming that $t\mapsto \Lambda(t)$ is a regular path, meaning that each crossing
with $W$ is regular, authors in \cite{LZ00} proved that  the Maslov index can be computed as follows
\begin{equation}\label{eq:CLM-index}
 \mu(W,\Lambda(t);[a,b])=\comorse(\Gamma(\Lambda(a),W,a))+\sum_{a<t<b}  \sgn
(\Gamma(\Lambda(t),W,t))-\morse (\Gamma(\Lambda(b),W,b)),
\end{equation}
where the sum runs  all over the  crossings $t\in(a,b)$ and where $\comorse,\morse$ denote the dimension of  positive and negative spectral subspaces whilst $\sgn=\comorse-\morse $ denote the signature.
\begin{rem}
We observe that by a sufficient small perturbation a $\mathscr C^1$-path of
Lagrangian subspaces with fixed end points is regular.
\end{rem}
In several interesting cases, for instance in the case of Lagrangian systems satisfying the classical Legendre convexity condition, the curve of Lagrangian subspaces induced by the action of the phase flow of the corresponding Hamiltonian system on a fixed Lagrangian subspace is a {\bf plus curve} with respect to the Dirichlet Lagrangian subspace $L_D$. This means that the local contribution to the Maslov index is provided by the dimension of the intersection of the path  wrt to $L_D$. So, in this specific case, we get that Equation~\eqref{eq:CLM-index} reduces to the following
\begin{equation*}
\mu(\ld ,\Lambda(t))=\text{dim}(\Lambda(a)\cap \ld )+\sum_{a<t<b} \text{dim}(\Lambda(t)\cap \ld ).
\end{equation*}
For more details we refer the interested reader to
\cite{RS93}, \cite{HO16}.

\textbf{Acknowledgments}. The authors thanks Prof. Xijun Hu and Prof. Guowei Yu for
their suggestions and  interesting discussions about  the index theory and the dynamics of  the gravitational $n$-body problem.

\vskip2truecm

\noindent

\begin{flushleft}
Prof. Alessandro Portaluri\\
DISAFA\\
Università degli Studi di Torino\\
Largo Paolo Braccini 2 \\
10095 Grugliasco, Torino\\
Italy\\
\medskip
Visiting professor of Mathematics\\
New York University Abu Dhabi\\
Saadiyat Marina District - Abu Dhabi\\
Emirates\\
Website: \texttt{https://sites.google.com/view/alessandro-portaluri/}\\
E-mail: \texttt{alessandro.portaluri@unito.it}\\
E-mail: \texttt{ap9453@nyu.edu}\\
\end{flushleft}

\vskip1truecm

\begin{flushleft}
Prof. Yuwei Ou \\
Department of mathematics  \\
Shandong University\\
Jinan, 250100, P. R. China \\
E-mail: \texttt{ywou@sdu.edu.cn}
\end{flushleft}

\end{document}